\documentclass[reqno]{amsart}

\usepackage[sorting=none, sorting=nyt, maxbibnames=99, backend=biber]{biblatex}

\renewbibmacro{in:}{}
\usepackage{amssymb}
\usepackage{calrsfs}
\usepackage{graphics,graphicx, mathrsfs}
\usepackage{enumerate}
\usepackage{enumitem}
\usepackage{url}
\usepackage{xcolor}
\definecolor{vio}{rgb}{0.54, 0.17, 0.89}
\usepackage[ruled, lined, linesnumbered, longend]{algorithm2e}
\usepackage{hyperref}
\usepackage{titlesec}

\newtheorem{theorem}{Theorem}[section]
\newtheorem{lemma}[theorem]{Lemma}

\newtheorem{proposition}[theorem]{Proposition}

\numberwithin{equation}{section}

\theoremstyle{remark}
\newtheorem*{remark}{Remark}

\titleformat{\section}
  {\normalfont\large\bfseries\centering}{\thesection}{1em}{}

\titleformat{\subsection}
  {\normalfont\bfseries}{\thesubsection}{1em}{}

\DeclareMathOperator{\PP}{\mathcal{P}}

\def\reals{\hbox{\rm I\kern-.18em R}}
\def\complexes{\hbox{\rm C\kern-.43em
\vrule depth 0ex height 1.4ex width .05em\kern.41em}}
\def\field{\hbox{\rm I\kern-.18em F}} 

\let\svthefootnote\thefootnote
\newcommand\freefootnote[1]{%
  \let\thefootnote\relax%
  \footnotetext{#1}%
  \let\thefootnote\svthefootnote%
}

\newenvironment{section*}[2][A]{
  \section*{#2}
  \renewcommand\thesection{#1}
  \setcounter{theorem}{0}}{}

\allowdisplaybreaks

\begin{document}

\title[Almost primes between all squares]{Almost primes between all squares}

\author{Adrian W. Dudek and Daniel R. Johnston}
\address{School of Mathematics and Physics, The University of Queensland}
\email{a.dudek@uq.edu.au}
\address{School of Science, UNSW Canberra, Australia}
\email{daniel.johnston@unsw.edu.au}
\date\today

\begin{abstract}
    We prove that for all $n\geq 1$ there exists a number between $n^2$ and $(n+1)^2$ with at most 4 prime factors. This is the first result of this kind that holds for every $n\geq 1$ rather than just sufficiently large $n$. Our approach relies on a recent computation by Sorenson and Webster, along with an explicit version of the linear sieve. As part of our proof, we also prove an explicit version of Kuhn's weighted sieve. This is done for generic sifting sets to enhance the future applicability of our methods.
\end{abstract}

\maketitle

\freefootnote{\textit{Corresponding author}: Daniel Johnston (daniel.johnston@unsw.edu.au).}
\freefootnote{\textit{Affiliation}: School of Science, The University of New South Wales Canberra, Australia.}
\freefootnote{\textit{Key phrases}: Legendre's conjecture, sieve methods, linear sieve, explicit results.}
\freefootnote{\textit{2020 Mathematics Subject Classification}: 11N36, 11N05.}

\section{Introduction}
\subsection{Overview}
One of the most famous unsolved problems in number theory is Legendre's conjecture, which asks whether there is always a prime between consecutive squares $n^2$ and $(n+1)^2$ for all $n\geq 1$. Even under the Riemann hypothesis, Legendre's conjecture is still open, and in 1912  Landau called this problem ``unattackable", listing it amongst three other major unsolved problems in the theory of prime numbers. Further historical notes on Landau's problems can be found in \cite{pintz2009landau}.

Over the last century however, there has been significant progress towards Legendre's conjecture. A classical 1937 result of Ingham \cite{ingham1937difference} implies that there exists a prime between consecutive cubes $n^3$ and $(n+1)^3$ for sufficiently large $n$. Then, after a series of improvements, it is now known that there exists a prime between $n^{2.106}$ and $(n+1)^{2.106}$ for all sufficiently large $n$ (see \cite{baker2001difference}). It is much more difficult however, to obtain such a result for all $n\geq 1$. In particular, by making Ingham's method explicit, the first author was able to prove the following result in 2016.

\begin{theorem}[{\cite[Theorem 1.1.]{dudek2016explicit}}]\label{cubethm}
    There is a prime between $n^3$ and $(n+1)^3$ for all $n\geq\exp(\exp(33.3))$.
\end{theorem}
Here, the impractically large $\exp(\exp(33.3))\approx 10^{10^{14}}$ was the result of estimates regarding the zeros of the Riemann zeta function, which have only been improved marginally in recent years (see e.g.\ \cite{cully2023primes,mossinghoff2024explicit}). 

In another direction, we can ask whether there exists an integer $k>0$ such that we can always find an integer $a$ with
\begin{equation}\label{almostlegendreeq}
    n^2<a<(n+1)^2\quad\text{and}\quad\Omega(a)\leq k.
\end{equation}
Here, as usual, $\Omega(a)$ is the number of prime factors of $a$ counting multiplicity, and from here onwards we will use the term \emph{almost prime} to refer to a positive integer with a small number of prime factors.

In 1920, Brun \cite[pp. 24--25]{brun1920crible} proved that for sufficiently large $n$, one could take $k=11$ in \eqref{almostlegendreeq}. After the rich development of sieve methods in the mid-20th century, this was eventually improved to $k=2$ by Chen \cite{chen1975distribution} in 1975, again for sufficiently large $n$.

The aim of this paper is to obtain the first explicit result of the form \eqref{almostlegendreeq} that holds for all $n\geq 1$. In particular, we prove the following.

\begin{theorem}\label{mainthm}
    For all $n\geq 1$, there exists $a\in(n^2,(n+1)^2)$ with $\Omega(a)\leq 4$.
\end{theorem}
To prove Theorem \ref{mainthm}, we use an explicit version of Kuhn's weighted sieve, combined with a recent computational verification of Legendre's conjecture by Sorenson and Webster \cite{sorenson2025}. It is likely that our method could be refined by using more elaborate sieve weighting procedures. However, it appears difficult to reduce the number of prime factors in Theorem \ref{mainthm} whilst also maintaining the range $n\geq 1$. In this regard, we note that a rough calculation shows our method would only give $k=3$ in \eqref{almostlegendreeq} for $n^2\geq 10^{50}$. 

With small changes, one could also use our method to detect almost primes in intervals besides $(n^2,(n+1)^2)$. In particular, using our approach one should, for all $n\geq 1$, be able to prove the existence of a number with at most 3 prime factors in $(n^3,(n+1)^3)$, or a number with at most 2 prime factors in $(n^4,(n+1)^4)$.

Finally, we note that compared to Theorem \ref{cubethm}, our result holds for a far more practical range of $n$. The main reason for this is that whilst Theorem \ref{cubethm} was proven using estimates pertaining to the Riemann zeta function, our argument is primarily sieve-theoretic. This highlights the nascent capabilities of using sieves for explicit results in number theory, an area for which they are seldom employed in place of other classical methods.

\subsection{Paper outline}
An outline of the rest of the paper is as follows. In Section \ref{prelimsect}, we set up the required preliminaries for sieving and also state a recent explicit version of the linear sieve due to Bordignon, Starichkova and the second author \cite[Theorem 6]{BJV24}. In Section \ref{Kuhnsect} we then give an explicit version of Kuhn's weighted sieve that will form the basis of our argument. Finally, in Section \ref{mainsect} we prove Theorem \ref{mainthm}. An appendix is also included that details the calculations required to obtain explicit values of ``$Q$" and ``$\varepsilon$"  appearing in the conditions for the linear sieve.

\section{Preliminaries to sieving}\label{prelimsect}
In sieve-theoretic problems one requires a finite set of positive integers $A$, and an infinite set of primes $\mathcal{P}$. The aim is then to remove (or ``sift out") all elements of $A$ with prime divisors in $\mathcal{P}$ that are less than some number $z$. For our application, we will let $\mathcal{P}$ be the set of all primes, and
\begin{equation}\label{Adef}
    A=A(N):=\mathbb{Z}\cap(N,N+2\sqrt{N}).
\end{equation}
Setting $N=n^2$, one sees that any $a\in A$ is contained in the set $(n^2,(n+1)^2)$. 

Throughout we may take $N>1.98\cdot 10^{28}$, owing to the following consequence of a computation by Sorenson and Webster \cite{sorenson2025}.
\begin{lemma}\label{complem}
    For all $n^2\leq 1.98\cdot 10^{28}$, there exists a number with at most $3$ prime factors in the interval $(n^2,(n+1)^2)$. 
\end{lemma}
\begin{proof}
    For $n^2\leq 4.97\cdot 10^{27}$, there exists a prime in both the intervals $(n^2,n(n+1))$ and $(n(n+1),(n+1)^2)$ by the computations in \cite{sorenson2025}. Hence, it suffices to consider $n^2\in(4.97\cdot 10^{27},1.98\cdot 10^{28}]$. We will show that for such a value of $n^2$, there exists an integer $4p\in(n^2,(n+1)^2)$, where $p$ is prime. Here $\Omega(4p)=3$.

    First consider the case where $n$ is even. Then by the computations in \cite{sorenson2025}, there exists a prime $p$ such that
    \begin{equation*}
        \left(\frac{n}{2}\right)^2<p<\frac{n}{2}\left(\frac{n}{2}+1\right)\quad\implies\quad n^2<4p<(n+1)^2,
    \end{equation*}
    as required. Similarly, in the case where $n$ is odd, there exists a prime $p$ such that
    \begin{equation*}
        \left(\frac{n+1}{2}-1\right)\frac{n+1}{2}<p<\left(\frac{n+1}{2}\right)^2\quad\implies \quad n^2-1<4p<(n+1)^2.
    \end{equation*}
    However, since $4p$ cannot be a square, $n^2<4p<(n+1)^2$ as required. 
\end{proof}
To prove the existence of almost primes in $A$ we require estimates for the sifting function 
\begin{equation}\label{SAPzeq}
    S(A,\mathcal{P},z):=\left|A\setminus\bigcup_{p\mid P(z)}A_p\right|,
\end{equation}
where $p$ is prime, $z\geq 2$, 
\begin{equation}\label{ApPzeq}
    A_p:=\{a\in A:p\mid a\}\quad\text{and}\quad P(z):=\prod_{\substack{p<z\\p\in\mathcal{P}}}p.
\end{equation}
This will be done using the following explicit version of the linear sieve due to Bordignon, Starichkova and the second author \cite{BJV24}. Here, we have slightly simplified their result for readability.
\begin{lemma}[{Explicit version of the linear sieve \cite[Theorem 6]{BJV24}}]\label{linsievelem}
    Let $A$ be a finite set of positive integers and $\mathcal{P}$ be an infinite set of primes such that no element of $A$ is divisible by a prime in the complement $\overline{\mathcal{P}}=\{p\ \text{prime}:p\notin\mathcal{P}$\}. Let $P(z)$ and $S(A,\mathcal{P},z)$ be as in \eqref{ApPzeq} and \eqref{SAPzeq} respectively. For any square-free integer $d$, let $g(d)$ be a multiplicative function with
    \begin{equation*}
        0\leq g(p)<1\ \text{for all}\ p\in\mathcal{P},
    \end{equation*}
    and
    \begin{equation*}
        A_d:=\{a\in A:d\mid a\}\quad\text{and}\quad r(d):=|A_d|-|A|g(d).
    \end{equation*}
    Let $\mathcal{Q}\subseteq\mathcal{P}$, and $Q=\prod_{q\in\mathcal{Q}}q$. Suppose that, for some $\varepsilon$ satisfying $0<\varepsilon\leq 1/74$, the inequality
    \begin{equation}\label{epseq}
        \prod_{\substack{p\in\mathcal{P}\setminus\mathcal{Q}\\ u\leq p<z}}\left(1-g(p)\right)^{-1}<(1+\varepsilon)\frac{\log z}{\log u},
    \end{equation}
    holds for all $1<u<z$. Then, for any $D\geq z$ we have the upper bound
    \begin{equation}\label{suppereq}
        S(A,\mathcal{P},z)<|A|V(z)\cdot\left(F(s)+\varepsilon C_1(\varepsilon)e^2h(s)\right)+R
    \end{equation}
    and for any $D\geq z^2$ we have the lower bound
    \begin{equation}\label{slowereq}
        S(A,\mathcal{P},z)>|A|V(z)\cdot (f(s)-\varepsilon C_{2}(\varepsilon) e^2h(s))-R,
    \end{equation}
    where
    \begin{equation*}
        s=\frac{\log D}{\log z},
    \end{equation*}
    \begin{equation} \label{eq-def-h(s)}
        h(s):= 
        \begin{cases} 
            e^{-2} & 1\le s \le 2, \\
            e^{-s} & 2\le s \le 3, \\
            3s^{-1}e^{-s} & s\ge 3,
        \end{cases} 
    \end{equation}
    $F(s)$ and $f(s)$ are the two functions defined by the following delay differential equations:
    \begin{equation} \label{def-f(s)-F(s)}
        \begin{cases}
            F(s) = \frac{2 e^{\gamma}}{s}, \quad f(s) = 0 &\text{for } 0 < s \le 2,\\
            (sF(s))' = f(s - 1), \quad (sf(s))' = F(s-1) &\text{for } s \ge 2,
        \end{cases}
    \end{equation}
    $C_1(\varepsilon)$ and $C_2(\varepsilon)$ come from \cite[Table 1]{BJV24},
    \begin{equation} \label{eq: def-X_A}
        V(z):=\prod_{p\mid P(z)}(1-g(p)),
    \end{equation}
    and the remainder term is
    \begin{equation}\label{Rdef}
        R:=\sum_{\substack{d|P(z)\\ d<QD}}|r(d)|.
    \end{equation}
\end{lemma}
For our application, with $A$ as in \eqref{Adef}, we have
\begin{equation*}
    2\sqrt{N}-1\leq |A|<2\sqrt{N}\ \text{for all}\ N\geq 1.
\end{equation*}
We also choose $g(d)=1/d$, so that
\begin{equation*}
    r(d)=|A_d|-\frac{|A|}{d}
\end{equation*}
and by \cite[Theorem 7]{rosser1962approximate},
\begin{equation*}
    V(z)=\prod_{p<z}\left(1-\frac{1}{p}\right)>\frac{e^{-\gamma}}{\log z} \left(1-\frac{1}{2 (\log z)^2}\right)
\end{equation*}
for all $z\geq 285$. Since $A$ is a set of consecutive integers, $|r(d)|\leq 1$, so that the sieve lower bound in \eqref{slowereq} yields the following.
\begin{lemma}\label{lowerlem}
    Let $A$ be as in \eqref{Adef}, $\PP$ be the set of all primes, and $z\geq 285$. Then, provided $D\geq z^2$ and $f(s)>\varepsilon C_2(\varepsilon) e^2h(s)$, we have
    \begin{equation}\label{lowereq}
        S(A,\PP,z)>\frac{e^{-\gamma}(2\sqrt{N}-1)}{\log z} \left(1-\frac{1}{2 (\log z)^2}\right)\cdot(f(s)-\varepsilon C_2(\varepsilon)e^2h(s))-\sum_{d<QD}\mu^2(d),
    \end{equation}
    where $s$, $f(s)$, $\varepsilon$, $C_2(\varepsilon)$, $h(s)$ and $Q$ are as in Lemma \ref{linsievelem}.
\end{lemma}
Suitable corresponding values of $Q$ and $\varepsilon$ are derived in the Appendix (Proposition \ref{PrimeProductBound}). Directly from \eqref{def-f(s)-F(s)}, we also have the following expressions for $F(s)$ and $f(s)$ when $1\leq s\leq 3$ and $2\leq s\leq 4$ respectively.
\begin{lemma}\label{flem}
    Let $F(s)$ and $f(s)$ be as in \eqref{def-f(s)-F(s)}. Then, for $1\leq s\leq 3$,
    \begin{equation*}
        F(s)=\frac{2e^{\gamma}}{s},
    \end{equation*}
    and for $2\leq s\leq 4$,
    \begin{equation*}
        f(s)=\frac{2e^{\gamma}\log(s-1)}{s}.
    \end{equation*}
\end{lemma}
More complicated expressions also exist for larger $s$ (see e.g.\ \cite[Lemma 2]{cai2008chen}). However, for our purposes the ranges of $s$ above are sufficient.

From Lemma \ref{lowerlem}, one can readily detect almost primes in $A$. In particular, if
\begin{equation*}
    X=N+2\sqrt{N}\ \text{and}\ z=X^{1/(k+1)},
\end{equation*}
then a positive lower bound in \eqref{lowereq} would imply the existence of elements $a\in A$ whereby all prime factors of $a$ are greater than $X^{1/(k+1)}$. Since any $a\in A$ satisfies $a<X$, it therefore follows that such an $a$ could only have at most $k$ prime factors.

This basic approach however, is rather inefficient as it sifts out many other desired elements of $A$ with a mix of small and large prime factors. Consequently, it leads to a slightly worse result than what is given in Theorem \ref{mainthm}. Namely, rough calculations by the authors give no better than $k=5$ for all $N>1.98\cdot 10^{28}$. Therefore, in order to obtain Theorem \ref{mainthm}, we must apply a \emph{weighted} sieve, which will use a combination of the upper and lower bounds for $S(A,\mathcal{P},z)$ in \eqref{linsievelem}. An explicit version of such a sieve method will be the focus of the next section. 

\section{An explicit version of Kuhn's weighted sieve}\label{Kuhnsect}
In this section we give an explicit version of a weighted sieve originally due to Kuhn \cite{kuhn1954neue}. Instead of simply sifting our set $A$ so that all remaining $a\in A$ have large prime factors, we will also keep $a\in A$ where at most one prime factor of $a$ is much smaller than the rest.

To do this, we introduce two parameters $k_1\in\mathbb{R}$ and $k_2\in\mathbb{Z}$ with $k_1\geq k_2\geq 2$. We then set
\begin{equation}\label{Xdef}
    X:=\max(A)
\end{equation}
and 
\begin{equation}\label{zyeq}
    z=X^{1/k_1}\ \text{and}\ y=X^{1/k_2}. 
\end{equation}
With $P(z)$ as in \eqref{ApPzeq}, we consider for any $a\in A$ with $(a,P(z))=1$, the ``weight" function
\begin{equation*}
    w(a)=1-\frac{1}{2}\sum_{\substack{q\in\PP\\ z\leq q<y\\q^{\ell}\| a}}\ell.
\end{equation*}
Note that $w(a)\leq 1$ and $w(a)=1$ if and only if $a$ has no prime divisors in $[z,y)$.

For $a$ as above, we write $a=p_1\cdots p_r\cdots p_{r+s}$ for the prime factorisation of $a$, where
\begin{equation*}
    z\leq p_1\leq\ldots\leq p_r< y\leq p_{r+1}\leq \ldots\leq p_{r+s}. 
\end{equation*}
Since $y=X^{1/k_2}$, we must have $s\leq k_2-1$.

Now suppose that $w(a)>0$. Since 
\begin{equation*}
    \frac{1}{2}\sum_{\substack{q\in\PP\\z\leq q<y\\q^\ell\| a}}\ell =\frac{r}{2},
\end{equation*}
we have $r\in\{0,1\}$. Therefore, $w(a)>0$ implies $r+s\leq k_2$. So defining
\begin{equation}\label{rkdef}
    r_k(A):=|\{a\in A:\Omega(a)\leq k\}|,
\end{equation}
we have
\begin{equation}\label{rkeq}
    r_{k_2}(A)\geq\sum_{a\in A}w(a).
\end{equation} 
The following lemma gives a bound for the sum in \eqref{rkeq}. For now, this will be done with generic sifting sets $A$ and $\PP$. Note that we also include a condition \eqref{q2con}, which is satisfied by most sifting problems (cf.\ \cite[page 253]{halberstam1974sieve}).
\begin{lemma}\label{kuhnlem}
    Let $A$ be a finite set of positive integers and $\mathcal{P}$ be an infinite set of primes such that no element of $A$ is divisible by a prime in $\overline{\mathcal{P}}$. Let $X$, $z$ and $y$ be as in \eqref{Xdef} and \eqref{zyeq}. Suppose further that the condition
    \begin{equation}\label{q2con}
        \sum_{\substack{z\leq q<y\\ q\in\mathcal{P}}}|A_{q^2}|\leq c_1\frac{|A|\log|A|}{z}+c_2y,
    \end{equation}
    holds for some constants $c_1,c_2>0$. Then,
    \begin{equation}\label{kuhneq}
        r_{k_2}(A)\geq S(A,\PP,z)-\frac{1}{2}\sum_{\substack{q\in\PP\\z\leq q<y}}S(A_q,\PP,z)-\frac{k_1 c_1|A| \log|A|}{2 z} -\frac{c_2 y}{2 \log z},
    \end{equation}
    where $r_{k_2}(A)$ is as in \eqref{rkdef}. 
\end{lemma}
\begin{proof}
    From \eqref{rkeq}, we have
    \begin{equation}\label{rkeq2}
        r_{k_2}(A)\geq \sum_{\substack{n\in A\\ (n,P(z))=1}}w(n)=\sum_{\substack{n\in A\\ (n,P(z))=1}}1-\frac{1}{2}\sum_{\substack{n\in A\\ (n,P(z))=1}}\sum_{\substack{q \in \mathcal{P} \\z\leq q<y\\q^\ell\| n}}\ell.
    \end{equation}
    Now,
    \begin{equation*}
        \sum_{\substack{n\in A\\ (n,P(z))=1}}1=S(A,\PP,z).
    \end{equation*}
    Next, we split the remaining term in \eqref{rkeq2} into two parts as follows.
    \begin{equation*}
        \frac{1}{2}\sum_{\substack{n\in A\\ (n,P(z))=1}}\sum_{\substack{q \in \mathcal{P} \\ z\leq q<y\\q^\ell\| n}}\ell=\frac{1}{2}\sum_{\substack{n\in A\\ (n,P(z))=1}}\sum_{\substack{q \in \mathcal{P} \\ z\leq q<y\\q\mid n}}1+\frac{1}{2}\sum_{\substack{n\in A\\ (n,P(z))=1}}\sum_{\substack{q \in \mathcal{P} \\ z\leq q<y\\q^\ell\| n\\ \ell\geq 2}}(\ell-1).
    \end{equation*}
    The first of these double sums is simply
    \begin{equation*}
        \frac{1}{2}\sum_{\substack{n\in A\\ (n,P(z))=1}}\sum_{\substack{q \in \mathcal{P} \\ z\leq q<y\\q\mid n}}1=\frac{1}{2}\sum_{\substack{q \in \mathcal{P} \\ z\leq q<y}}S(A_q,\PP,z).
    \end{equation*}
    As for the second double sum, we note that
    \begin{equation*}
        \frac{1}{2}\sum_{\substack{n\in A\\ (n,P(z))=1}}\sum_{\substack{q \in \mathcal{P} \\ z\leq q<y\\q^\ell\| n\\ \ell\geq 2}}(\ell-1) < \frac{1}{2}\sum_{\substack{n\in A\\ (n,P(z))=1}}\sum_{\substack{q \in \mathcal{P} \\z\leq q<y\\q^2 | n}} k_1
    \end{equation*}
    since
    $$\ell-1 < \frac{\log n}{\log q} \leq \frac{\log X}{\log z}=k_1,$$
    and counting all $q^\ell \|n$ for $l \geq 2$  is equivalent to counting all $q^2 | n$. We then have
    \begin{equation*}
        \frac{1}{2}\sum_{\substack{n\in A\\ (n,P(z))=1}}\sum_{\substack{q \in \mathcal{P} \\z\leq q<y\\q^2 | n}} k_1 = \frac{k_1}{2} \sum_{\substack{q \in \mathcal{P} \\z\leq q<y}} |A_{q^2}|
    \end{equation*}
    and a direct application of condition \eqref{q2con} finishes the proof.
\end{proof}
\begin{remark}
    One could more generally define a weight function 
    \begin{equation*}
        w_b(a)=1-\frac{1}{b+1}\sum_{\substack{q\in\PP\\ z\leq q<y\\q^{\ell}\| a}}\ell,
    \end{equation*}
    where $b\geq 1$ is a fixed integer. This would instead allow at most $b$ prime factors less than $y$ and thereby give a lower bound for $r_{k_2+(b-1)}(A)$. However, for most applications (including ours), the choice $b=1$ is optimal.
\end{remark}
With a view to apply Lemma \ref{kuhnlem} for our application, we can use the lower bound on $S(A,\PP,z)$ in Lemma \ref{lowerlem}. Then, to bound the sum over $q$ in \eqref{kuhneq}, we will require the following lemmas.
\begin{lemma}[{\cite[\S 1.3.5, Lemma 1 (ii)]{greaves2013sieves}}]\label{sumintlem}
    Let $f(t)$ be a positive, monotone function defined for $z\leq t\leq y$ with $f'(t)$ piecewise continuous on $[z,y]$, and $c(n)$ be an arithmetic function satisfying 
    \begin{equation*}
        \sum_{x\leq n<w}c(n)\leq g(w)-g(x)+E,
    \end{equation*}
    for some constant $E$ whenever $z\leq x < w \leq y$. Then,
    \begin{equation*}
        \sum_{z\leq n<y}c(n)f(n)\leq\int_z^yf(t)g'(t)\mathrm{d}t+E\max(f(z),f(y)).
    \end{equation*}
\end{lemma}
\begin{lemma}[{\cite[Theorem 7]{rosser1962approximate}}]\label{Vlem}
    For all $x>1$,
    \begin{equation*}
        \prod_{p\leq x}\left(1-\frac{1}{p}\right)<\frac{e^{-\gamma}}{\log x}\left(1+\frac{1}{2(\log x)^2}\right).
    \end{equation*}
\end{lemma}
\begin{lemma}[{\cite[Corollary 1]{vanlalngaia2017explicit}}]\label{mertenintlem}
    For any $b>a>1000$,
    \begin{equation*}
        \sum_{a\leq p< b}\frac{1}{p}<\log\log b-\log\log a+\frac{5}{(\log a)^3}.
    \end{equation*}
\end{lemma}

We now give an upper bound for the sum over $S(A_q,\PP,z)$ that is of the same order as our lower bound for $S(A,\PP,z)$ in Lemma \ref{lowerlem}.  

\begin{proposition} \label{upperprop}
    Let $A=A(N)$ be as in \eqref{Adef}, $X=\lfloor N+2\sqrt{N}\rfloor$, and $\mathcal{P}$ be the set of all primes. Let $y$ and $z$ be as in \eqref{zyeq} and assume that $y>z>1000$ and $2<k_1\leq 8$. Let $\alpha$ be a real number such that
    \begin{equation}\label{alphabounds}
        0<\alpha<\frac{1}{2}-\frac{1}{k_1}-\frac{1}{k_2}
    \end{equation}
    and set
    \begin{equation*}
        k_{\alpha }=k_1\left(\frac{1}{2}-\frac{1}{k_2}-\alpha \right).
    \end{equation*}
    Then, with $Q$, $\varepsilon$ and $C_1(\varepsilon)$ as in Lemma \ref{linsievelem}, we have
    \begin{equation*}
        \sum_{z\leq q<y}S(A_q,\PP,z)\leq k_1e^{-\gamma}\left(1+\frac{k_1^2}{2(\log X)^2}\right)\left(M_1(X)+M_2(X)\right)+\mathcal{E}(X),
    \end{equation*}
    where
    \begin{align}
        M_1(X):=&\frac{2\sqrt{N}}{\log X}\left[\frac{2e^{\gamma}}{k_1}\left(\frac{\log\left(\frac{k_1-2k_1\alpha -2}{k_2-2k_2\alpha -2}\right)}{\left(\frac{1}{2}-\alpha \right)}+\frac{5k_1^4}{k_{\alpha }(\log X)^3}\right)\right.\notag\\
        &\qquad\qquad\qquad\left.+\varepsilon C_1(\varepsilon)e^2h(k_{\alpha })\left(\log\left(\frac{k_1}{k_2}\right)+\frac{5k_1^3}{(\log X)^3}\right)\right],\label{m1def}
    \end{align}
    \begin{equation}\label{m2def}
        M_2(X):=\frac{y}{\log X}\left(\frac{2e^{\gamma}}{k_{\alpha }}+\varepsilon C_1(\varepsilon)e^2h(k_{\alpha })\right)
    \end{equation}
    and
    \begin{equation}\label{Edef}
        \mathcal{E}(X):=QX^{\frac{1}{2}-\alpha }\left(\log\left(\frac{k_1}{k_2}\right)+\frac{5k_1^3}{(\log X)^3}\right).
    \end{equation}
\end{proposition}
\begin{proof}
    We start by defining 
    \begin{equation} \label{eq: def-Q-D2-D2q-sq}
        D_q:=\frac{X^{\frac{1}{2}-\alpha }}{q}\quad\text{and}\quad s_q:=\frac{\log D_q}{\log z}.
    \end{equation}
    Now, for any prime $q\in[z,y)$, let
    \begin{equation*}
        r_q(d)=|A_{qd}|-\frac{|A_q|}{d}.
    \end{equation*}
    By \eqref{suppereq} in Lemma \ref{linsievelem}, and Lemma \ref{Vlem},
    \begin{align}
        \sum_{z\leq q<y}S(A_q,\PP,z)&< k_1 e^{-\gamma}\left(1+\frac{k_1^2}{2(\log X)^2}\right)\sum_{z\leq q<y}|A_q|\left(\frac{F(s_q)+\varepsilon C_1(\varepsilon)e^2h(s_q)}{\log X}\right)\notag\\
        &\qquad\qquad +\sum_{z\leq q<y}\sum_{\substack{d\mid P(z)\\d<QD_q}}|r_q(d)|.\label{SAqexpandeq}
    \end{align}
    Here, we note that $s_q>1$ by \eqref{alphabounds}. We start by bounding the remainder term in \eqref{SAqexpandeq}. Since $A$ is a set of consecutive integers, $|r_q(d)|\leq 1$ and we have
    \begin{equation*}
        \sum_{z\leq q<y}\sum_{\substack{d\mid P(z)\\d<QD_q}}|r_q(d)|\leq QX^{\frac{1}{2}-\alpha }\sum_{z\leq q<y}\frac{1}{q}.
    \end{equation*}
    By Lemma \ref{mertenintlem}, this is bounded above by
    \begin{equation}\label{remainderbound}
        QX^{\frac{1}{2}-\alpha }\left(\log\left(\frac{\log y}{\log z}\right)+\frac{5}{(\log z)^3}\right)=QX^{\frac{1}{2}-\alpha }\left(\log\left(\frac{k_1}{k_2}\right)+\frac{5k_1^3}{(\log X)^3}\right)=\mathcal{E}(X).
    \end{equation}
    We now bound the main term in \eqref{SAqexpandeq}. To begin with, we note that 
    \begin{equation}\label{Aqbound}
        |A_q|\leq\frac{2\sqrt{N}}{q}+1.
    \end{equation}
    Moreover, since $k_1\leq 8$, we have $s_q<3$ so that by Lemma \ref{flem}
    \begin{equation}\label{Fsqeq}
        F(s_q)=\frac{2e^{\gamma}}{s_q}=\frac{2e^{\gamma}\log X}{k_1\log D_q}.
    \end{equation}
    By \eqref{Aqbound} and \eqref{Fsqeq}, along with the fact that $s_q\geq k_{\alpha }$ and $D_q\geq D_y$,
    \begin{align}
        &\sum_{z\leq q<y}|A_q|\left(\frac{F(s_q)+\varepsilon C_1(\varepsilon)e^2h(s_q)}{\log X}\right)\notag\\
        &\qquad\leq 2\sqrt{N}\sum_{z\leq q<y}\frac{1}{q}\left(\frac{2e^{\gamma}}{k_1\log D_q}+\frac{\varepsilon C_1(\varepsilon)e^2h(k_{\alpha })}{\log X}\right)\notag\\
        &\qquad\qquad\qquad\qquad\qquad+\sum_{z\leq q<y}\left(\frac{2e^{\gamma}}{k_1\log D_{y}}+\frac{\varepsilon C_1(\varepsilon)e^2h(k_{\alpha })}{\log X}\right).\label{SAqmain}
    \end{align}
    The second term in \eqref{SAqmain} can be simply bounded as
    \begin{equation}\label{secondSAq}
        \sum_{z\leq q<y}\left(\frac{2e^{\gamma}}{k_1\log D_{y}}+\frac{\varepsilon C_1(\varepsilon)e^2h(k_{\alpha })}{\log X}\right)\leq \frac{y}{\log X}\left(\frac{2e^{\gamma}}{k_{\alpha }}+\varepsilon C_1(\varepsilon)e^2h(k_{\alpha })\right)=M_2(X),
    \end{equation}
    since
    \begin{equation*}
        D_y=X^{\frac{1}{2}-\alpha -\frac{1}{k_2}}.
    \end{equation*}
    For the first term in \eqref{SAqmain} we again note that
    \begin{equation}\label{merteninteq}
        \sum_{z\leq q<y}\frac{1}{q}<\log\left(\frac{k_1}{k_2}\right)+\frac{5k_1^3}{(\log X)^3}
    \end{equation}
    by Lemma \ref{mertenintlem}. We then evaluate the sum
    \begin{equation*}
        \sum_{z\leq q<y}\frac{1}{qD_q}
    \end{equation*}
    using Lemma \ref{sumintlem}. In particular, applying Lemma \ref{sumintlem} with $f(t)=1/\log(D_t)$, $g(t)=\log\log t$, 
    \begin{equation*}
        c(n)=
        \begin{cases}
            1/n,&\text{if $n$ is prime},\\
            0,&\text{otherwise}
        \end{cases}
    \end{equation*}
    and $E=5k_1^3/(\log X)^3$ (as a consequence of Lemma \ref{mertenintlem}),
    \begin{align}
        \sum_{z \le q < y}\frac{1}{q\log D_q}&\le\int_z^y\frac{1}{t\log t \log D_t}\mathrm{d}t+\frac{5k_1^3}{(\log X)^3}\frac{1}{\log D_y}\notag\\
        & = \frac{\log\left(\frac{k_1-2k_1\alpha -2}{k_2-2k_2\alpha -2}\right)}{\left(\frac{1}{2}-\alpha \right) \log X}+\frac{5k_1^4}{k_{\alpha }(\log X)^4},\label{logdtint}
    \end{align}
    where we substituted $z=X^{1/k_1}$, $y=X^{1/k_2}$ and $D_q=X^{1/2-\alpha }/q$ to obtain the final equality. Using \eqref{merteninteq} and \eqref{logdtint} we therefore have
    \begin{align}\label{firstSAq}
        &2\sqrt{N}\sum_{z\leq q<y}\frac{1}{q}\left(\frac{2e^{\gamma}}{k_1\log D_q}+\frac{\varepsilon C_1(\varepsilon)e^2h(k_{\alpha })}{\log X}\right)\notag\\
        &\qquad<\frac{2\sqrt{N}}{\log X}\left[\frac{2e^{\gamma}}{k_1}\left(\frac{\log\left(\frac{k_1-2k_1\alpha -2}{k_2-2k_2\alpha -2}\right)}{\left(\frac{1}{2}-\alpha \right)}+\frac{5k_1^4}{k_{\alpha }(\log X)^3}\right)\right.\notag\\
        &\qquad\qquad\qquad\qquad\left.+\varepsilon C_1(\varepsilon)e^2h(k_{\alpha })\left(\log\left(\frac{k_1}{k_2}\right)+\frac{5k_1^3}{(\log X)^3}\right)\right]=M_1(X).
    \end{align}
    Combining \eqref{secondSAq} and \eqref{firstSAq} bounds \eqref{SAqmain} by $M_1(X)+M_2(X)$. Combining this with our estimate $\mathcal{E}(X)$ for the remainder term in \eqref{SAqexpandeq} then proves the proposition. 
\end{proof}

\section{Proof of Theorem \ref{mainthm}}\label{mainsect}
We now collate all our results to prove Theorem \ref{mainthm}. First however, we will give an explicit version of the condition \eqref{q2con} for our application. Throughout we choose $z=X^{1/8}$ and $y=X^{1/4}$ ($k_1=8$ and $k_2=4$). Although one could optimise further over the choice of $z=X^{1/k_1}$, we found $z=X^{1/8}$ to be a simple choice sufficient for our purposes.
\begin{lemma}\label{q2conlem}
    Let $A=A(N)$ be as in \eqref{Adef}, $\mathcal{P}$ be the set of all primes, $X~=~\max(A)$, $z=X^{1/8}$ and $y=X^{1/4}$. Then, for $N>1.98\cdot 10^{28}$,
    \begin{equation}\label{ourq2con}
        \sum_{\substack{z\leq q<y\\ q\in\mathcal{P}}}|A_{q^2}|\leq c_1\frac{|A|\log|A|}{z}+c_2y,
    \end{equation}
    with $c_1=0.01$ and $c_2=0.07$.
\end{lemma}
\begin{proof}
    Since $A$ is a set of consecutive integers, 
    \begin{eqnarray}
        \sum_{\substack{z \leq q < y \\ q \in \mathcal{P}}}|A_{q^2}| & \leq & \sum_{\substack{z \leq q < y \\ q \in \mathcal{P}}} \bigg(\frac{|A|}{q^2}+1\bigg) \notag\\
        & = & |A| \sum_{\substack{z \leq q < y \\ q \in \mathcal{P}}} \frac{1}{q^2} +\sum_{\substack{z \leq q < y \\ q \in \mathcal{P}}} 1 \notag\\
        & \leq & \frac{2.22 |A|}{z \log z} + \pi(y)\label{finalAq2eq}
    \end{eqnarray}
    where the last line follows from \cite[Lemma 7]{glasby2021most}. Now $X=\lfloor N+2\sqrt{N}\rfloor$, so we have
    \begin{equation}\label{zrange}
        z \geq (1.98\cdot 10^{28} + 2 \sqrt{1.98\cdot 10^{28}})^{1/8} > 3444
    \end{equation}
    and
    $$y \geq (1.98\cdot 10^{28} + 2 \sqrt{1.98\cdot 10^{28}})^{1/4} > 10^{7}.$$
    In such a range for $y$, we use \cite[Theorem 1]{rosser1962approximate} to get that
    $$\pi(y) < 1.1 \frac{y}{\log y}.$$ Hence, we can take
    $$c_2 = 0.07 > \frac{1.1}{\log (10^7)}$$
    in \eqref{ourq2con}. Then, from \eqref{finalAq2eq}, a permissible $c_1$ in \eqref{ourq2con} is such that
    $$\frac{2.22 |A|}{z \log z}\leq \frac{c_1 |A| \log |A|}{z}$$
    which rearranges to
    $$c_1 \geq \frac{2.22}{\log |A| \log z}.$$
    Therefore, using \eqref{zrange} and the fact that $|A|\geq 2\sqrt{N}-1$, we may take
    $$c_1 = 0.01$$
    as claimed.
\end{proof}
Using these values of $c_1$ and $c_2$, Theorem \ref{mainthm} can now be deduced from Lemma  \ref{kuhnlem} combined with Lemmas \ref{lowerlem} and \ref{upperprop}.
\begin{proof}[Proof of Theorem \ref{mainthm}]
    Let $A=A(N)$ be as in \eqref{Adef}, $\mathcal{P}$ be the set of all primes and $X=\max(A)=\lfloor N+2\sqrt{N}\rfloor$.  It suffices to show that $r_4(A)$, defined in \eqref{rkdef}, is positive for all $N>1.98\cdot 10^{28}$, since Lemma \ref{complem} covers $N\leq 1.98\cdot 10^{28}$. Let $z=X^{1/8}$ and $y=X^{1/4}$. Substituting Lemma \ref{q2conlem} into Lemma \ref{kuhnlem} gives 
    \begin{equation}\label{r4boundfirst}
        r_4(A)\geq S(A,\PP,z)-\frac{1}{2}\sum_{\substack{q\in\PP\\z\leq q<y}}S(A_q,\PP,z)-\frac{0.08 \sqrt{N}\log(2\sqrt{N})}{z} -\frac{0.56 y}{2 \log X},
    \end{equation}
    noting that $|A|\leq 2\sqrt{N}$. Here, for $N>1.98\cdot 10^{28}$,
    \begin{equation*}
        \frac{0.08 \sqrt{N}\log(2\sqrt{N})}{z}+\frac{0.56 y}{2 \log X}\leq\frac{0.051\sqrt{N}}{\log X}        
    \end{equation*}
    so we can reduce \eqref{r4boundfirst} to the simpler bound
    \begin{equation}\label{r4bound}
         r_4(A)\geq S(A,\PP,z)-\frac{1}{2}\sum_{\substack{q\in\PP\\z\leq q<y}}S(A_q,\PP,z)-\frac{0.051\sqrt{N}}{\log X}.
    \end{equation}
    We will now bound the first two terms in \eqref{r4bound} separately.\\

    \textbf{Lower bound for $S(A,\mathcal{P},z)$}: We begin by bounding $S(A,\mathcal{P},z)$. By Lemma \ref{lowerlem}, and the notation therein,
    \begin{equation}\label{lowereq2}
        S(A,\PP,z)>\frac{e^{-\gamma}(2\sqrt{N}-1)}{\log z} \left(1-\frac{1}{2 (\log z)^2}\right)\cdot(f(s)-\varepsilon C_2(\varepsilon)e^2h(s))-\sum_{d<QD}\mu^2(d),
    \end{equation}
    provided $D\geq z^2$ and $f(s)>\varepsilon C_2(\varepsilon)e^2h(s)$. Now,
    \begin{equation}\label{lowerbound1}
        \frac{e^{-\gamma}(2\sqrt{N}-1)}{\log z} \left(1-\frac{1}{2 (\log z)^2}\right)>1.1\frac{\sqrt{N}}{\log z}=8.8\frac{\sqrt{N}}{\log X}
    \end{equation}
    using that $N>1.98\cdot 10^{28}$, and
    \begin{equation}\label{zbound}
        z=X^{1/8}\geq (1.98\cdot 10^{28}+2\sqrt{1.98\cdot 10^{28}})^{1/8}>3444.
    \end{equation}
    Next, by Proposition \ref{PrimeProductBound}, we may take $Q=2$ and $\varepsilon=1.97\cdot 10^{-3}$ in \eqref{lowereq2}, and $C_2(\varepsilon)=122$ by \cite[Table 1]{BJV24}. Let $D\in[z^3,z^4]$ to be determined later. By Lemma~\ref{flem},
    \begin{equation*}
        f(s)=\frac{2e^{\gamma}\log(s-1)}{s}.
    \end{equation*} 
    noting that $s:=\log D/\log z\in [3,4]$. Hence, by the definition \eqref{eq-def-h(s)} of $h(s)$,
    \begin{equation}\label{csbound}
        f(s)-\varepsilon C_2(\varepsilon)e^2h(s)>\frac{2e^{\gamma}\log(s-1)}{s}-(1.97\cdot 10^{-3})\cdot 122e^2\cdot\frac{3e^{-s}}{s}>C(s),
    \end{equation}
    where
    \begin{equation}\label{cseq}
        C(s):=\frac{2e^{\gamma}\log(s-1)-0.73e^{2-s}}{s},
    \end{equation}
    which is positive for all $3\leq s\leq 4$. Thus, substituting \eqref{lowerbound1}, \eqref{csbound} and $Q=2$ into \eqref{lowereq2}, we have
    \begin{equation}\label{SAPzbound2}
        S(A,\mathcal{P},z)>  C(s)\frac{8.8\sqrt{N}}{\log X}-\sum_{d<2D}\mu^2(d).
    \end{equation}
    To bound the sum over $\mu^2(d)$, we note that by \cite[Lemma 4.6]{ramare2013explicit},
    \begin{equation}\label{buthebound}
        \sum_{n\leq x}\mu^2(n)\leq\frac{6}{\pi^2}x+0.5\sqrt{x}
    \end{equation}
    for all $x\geq 10$. Now, by \eqref{zbound},
    \begin{equation*}
        D\geq z^3 > 4\cdot 10^{10}.
    \end{equation*}
    Therefore, \eqref{buthebound} gives
    \begin{equation*}
        \sum_{d<2D}\mu^2(d)< 1.216D=1.216X^{s/8},
    \end{equation*}
    which, substituted into \eqref{SAPzbound2}, yields
    \begin{equation}\label{finalSApz}
        S(A,\mathcal{P},z)>C(s)\frac{8.8\sqrt{N}}{\log X}-1.216X^{s/8}.
    \end{equation}
    \ \\ \\
    \textbf{Upper bound for sum over $S(A_q,\mathcal{P},z)$:} Next we bound the sum over $S(A_q,\mathcal{P},z)$ appearing in \eqref{r4bound}. To do this, we apply Proposition \ref{upperprop} with $k_1=8$, $k_2=4$ and choose $\alpha=0.07$ (and thus $k_{\alpha}=1.44$). This yields
    \begin{equation}\label{SAQfirstbound}
        \sum_{z\leq q<y}S(A_q,\PP,z)\leq 8e^{-\gamma}\left(1+\frac{32}{(\log X)^2}\right)\left(M_1(X)+M_2(X)\right)+\mathcal{E}(X),
    \end{equation}
    with $M_1(X)$, $M_2(X)$ and $\mathcal{E}(X)$ defined in \eqref{m1def}, \eqref{m2def} and \eqref{Edef} respectively. As in the derivation for our bound \eqref{finalSApz} for $S(A,\mathcal{P},z)$, we let $Q=2$, $\varepsilon=1.97\cdot 10^{-3}$ and $C_1(\varepsilon)=121$, courtesy of Proposition \ref{PrimeProductBound} and \cite[Table 1]{BJV24}. So, for $N> 1.98\cdot 10^{28}$,
    \begin{equation}\label{leadingAqbound}
        8e^{-\gamma}\left(1+\frac{32}{(\log X)^2}\right)\leq 4.526
    \end{equation}
    and for our choice of parameters,
    \begin{align}
        M_1(X)&\leq \frac{2\sqrt{N}}{\log X}\left[\frac{e^{\gamma}}{4}\left(\frac{\log(3.4)}{0.43}+\frac{20480}{1.44(\log X)^2}\right)+0.24\left(\log\left(2\right)+\frac{2560}{(\log X)^3}\right)\right]\notag\\
        &\leq \frac{2.909\sqrt{N}}{\log X},\label{M1bound}\\
        M_2(X)&\leq\frac{X^{1/4}}{\log X}\left(\frac{2e^{\gamma}}{1.44}+0.24\right)\leq\frac{2.713X^{1/4}}{\log X}\label{M2bound},\\
        \mathcal{E}(X)&= 2\cdot X^{0.43}\left(\log\left(2\right)+\frac{2560}{(\log X)^3}\right)\leq 1.405X^{0.43}.\label{ecalbound}
    \end{align}
    Substituting \eqref{leadingAqbound}, \eqref{M1bound}, \eqref{M2bound} and \eqref{ecalbound} into \eqref{SAQfirstbound} gives
    \begin{equation}\label{finalSAq}
        \sum_{z\leq q<y}S(A_q,\mathcal{P},z)\leq\frac{13.167\sqrt{N}}{\log X}+\frac{12.28 X^{1/4}}{\log X}+1.405X^{0.43}\leq\frac{14.124\sqrt{N}}{\log X}.
    \end{equation}
    \ \\ \\
    \textbf{Combining bounds:}
    Substituting \eqref{finalSApz} and \eqref{finalSAq} into \eqref{r4bound}, we have
    \begin{equation}\label{finalr4eq}
        r_4(A)\geq\frac{(8.8C(s)-7.113)\sqrt{N}}{\log X}-1.216X^{s/8},
    \end{equation}
    with $C(s)$ as in \eqref{cseq}. Setting $s=3.3$ gives $C(s)>0.839$ so that \eqref{finalr4eq} yields $r_4(A)>0$ for all $N>1.98\cdot 10^{28}$ as required. 
\end{proof}

\section*{Acknowledgements}
We thank Timothy Trudgian for some helpful suggestions as well as Lachlan Dunn and Simon Thomas for identifying a couple of small errors and typos in an earlier version of this work. The second author's research was supported by an Australian Government Research Training Program (RTP) Scholarship.
\newpage 

\begin{section*}[A]{Appendix: Explicit values for $Q$ and $\varepsilon$ in the linear sieve}
\setcounter{equation}{0}
In this appendix we prove an explicit version of the bound \eqref{epseq} appearing in Lemma \ref{linsievelem}. This amounts to providing a value of $Q$ and $\varepsilon$ for use in Lemma \ref{lowerlem} and Proposition \ref{upperprop}. 

First we give some explicit bounds on the error in Mertens' third theorem.
\begin{lemma}\label{mertenlem}
    For $2\leq x\leq 4\cdot 10^9$,
    \begin{align}\label{mert1}
        e^\gamma\log x<\prod_{p\leq x}\left(1-\frac{1}{p}\right)^{-1}<e^\gamma\log x+\frac{2e^{\gamma}}{\sqrt{x}},
    \end{align}
    and for $x\geq\exp(22)$,
    \begin{equation}\label{mert2}
        e^\gamma\log x\left(\frac{1}{1+\frac{0.841}{(\log x)^3}}\right)\leq\prod_{p\leq x}\left(1-\frac{1}{p}\right)^{-1}\leq e^\gamma\log x\left(\frac{1}{1-\frac{0.841}{(\log x)^3}}\right),
    \end{equation}
    where $\gamma$ is the Euler-Mascheroni constant.
\end{lemma}
\begin{proof}
    To begin with, note that \eqref{mert2} is just \cite[Theorem 5]{vanlalngaia2017explicit} in the case $\epsilon=-1$. 
    
    Then, to prove \eqref{mert1} we used a direct computation. This computation took approximately 2.5 hours on a laptop with a 2.40 GHz processor. To begin with, we used the primesieve package in Python to compute and store all primes less than $10^8$. With $p_n$ denoting the $n$th prime number, we then verified that
    \begin{equation}\label{mertcompeq}
        e^\gamma\log p_{n+1}<\prod_{p\leq p_n}\left(1-\frac{1}{p}\right)^{-1}<e^\gamma\log p_n+\frac{2e^{\gamma}}{\sqrt{p_n}}
    \end{equation}
    for each\footnote{Some care needs to be taken at the end point here. We also had to compute the next prime after $10^8$ as to be able to substitute a value for $p_{n+1}$ into \eqref{mertcompeq}.} $p_n\leq 10^8$. Here, we note that if $x\in[p_n,p_{n+1}]$, then
    \begin{equation*}
        e^{\gamma}\log x\leq e^{\gamma}\log p_{n+1}\quad\text{and}\quad e^{\gamma}\log p_n+\frac{2e^{\gamma}}{\sqrt{p_n}}\leq e^\gamma\log x+\frac{2e^{\gamma}}{\sqrt{x}}. 
    \end{equation*}
    Hence, verifying \eqref{mertcompeq} for $p_n\leq 10^8$ also proves \eqref{mert1} for all $x\leq 10^8$. We repeated this process in intervals of length $10^8$ until \eqref{mert1} was proven for all $x\leq 4\cdot 10^9$.
\end{proof}
\begin{remark}
    The result \eqref{mert1} extends a classical computation by Rosser and Schoenfeld \cite[Theorem 23]{rosser1962approximate}. Note there is a minor error in \cite[Theorem 23]{rosser1962approximate} whereby it is claimed that \eqref{mert1} holds for all $0<x<2$, yet this is not true at e.g.\ $x=1.9$.
\end{remark}

We now prove the desired bound.
\begin{proposition} 
\label{PrimeProductBound}
    For all $u\geq 3$ and $z\geq 3024$, we have
    \begin{equation}\label{epsbound}
        \prod_{u\leq p<z}\left(1-\frac{1}{p}\right)^{-1}<(1+1.97\cdot 10^{-3})\frac{\log z}{\log u}.    
    \end{equation}
\end{proposition}
In the notation of Lemma \ref{linsievelem}, the bound \eqref{epsbound} allows us to take $\varepsilon=1.97\cdot 10^{-3}$ for $\mathcal{Q}=\{2\}$ and thus $Q=2$.
\begin{proof}[Proof of Proposition \ref{PrimeProductBound}]
    We split the proof into multiple cases for $z$.\\
    \\
    \textbf{Case 1:} $3024\leq z<12000$.\\
    For this range of $z$, we verify the lemma by a direct computation. To begin with, we compute the product
    \begin{equation*}
        I(x):=\prod_{p\leq x}\left(1-\frac{1}{p}\right)^{-1}
    \end{equation*}
    for all $0<x\leq 12000$. We now split into two further cases, depending on whether $u\geq 3$ is prime or not. Firstly, if $u\geq 3$ is not a prime number then
    \begin{equation}\label{ucase2}
        \prod_{u\leq p<z}\left(1-\frac{1}{p}\right)^{-1}\frac{\log u}{\log z}\leq\frac{I(\lfloor z\rfloor)}{I(\lceil u\rceil)}\frac{\log\lceil u\rceil}{\log(\lfloor z\rfloor)}.
    \end{equation}
    Computing the right-hand side of \eqref{ucase2} for all possible $3024\leq \lfloor z\rfloor<12000$ and $4\leq\lceil u\rceil\leq \lfloor z\rfloor$ gives a maximum value of $1+(1.904554\ldots)\cdot 10^{-3}$, occurring at $\lceil u\rceil=3298$ and $\lfloor z\rfloor=3947$.

    Now, suppose that $u\geq 3$ is a prime number. Then
    \begin{equation}\label{ucase3}
        \prod_{u\leq p<z}\left(1-\frac{1}{p}\right)^{-1}\frac{\log u}{\log z}\leq\frac{I(\lfloor z\rfloor)}{I(u-1)}\frac{\log u}{\log(\lfloor z\rfloor)},
    \end{equation}
    Computing the right-hand side of \eqref{ucase3} for all integers $\lfloor z\rfloor$ with $3024\leq\lfloor z\rfloor<12000$ and $u$ prime with $3\leq u<\lfloor z\rfloor$ gives a maximum value of $1+(1.967179\ldots)\cdot 10^{-3}$, occurring at $u=1423$ and $\lfloor z\rfloor=3947$. Each of the above computations took less than a minute on a laptop and verify that for all $3024\leq z<12000$, the inequality \eqref{epsbound} is satisfied.
    \\
    \\
    \textbf{Case 2:} $12000\leq z<4\cdot 10^9$.\\
    For this range of $z$, \eqref{mert1} gives
    \begin{equation*}
        \prod_{u\leq p<z}\left(1-\frac{1}{p}\right)^{-1}\leq\frac{\log z+\frac{2}{\sqrt{z}}}{\log u}=\frac{\log z}{\log u}\left(1+\frac{2}{\sqrt{z}\log z}\right)<\frac{\log z}{\log u}\left(1+1.95\cdot 10^{-3}\right),
    \end{equation*}
    which is slightly stronger than \eqref{epsbound}.\\
    \\
    \textbf{Case 3:} $z\geq 4\cdot 10^9$.\\
    For this range of $z$, if $u\leq 4\cdot 10^9$, then \eqref{mert1} and \eqref{mert2} yield
    \begin{align*}
        \prod_{u\leq p<z}\left(1-\frac{1}{p}\right)^{-1}\leq\frac{\log z}{\log u}\left(\frac{1}{1-\frac{0.841}{(\log z)^3}}\right)<\frac{\log z}{\log u}\left(\frac{1}{1-\frac{0.841}{(22)^3}}\right)<\frac{\log z}{\log u}\left(1+10^{-4}\right)
    \end{align*}
    since $4\cdot 10^9>\exp(22)$. Otherwise, if $u>4\cdot 10^9$, we apply \eqref{mert2} to obtain
    \begin{align*}
        \prod_{u\leq p<z}\left(1-\frac{1}{p}\right)^{-1}\leq\frac{\log z}{\log u}\left(\frac{1+\frac{0.841}{(\log u)^3}}{1-\frac{0.841}{(\log z)^3}}\right)<\frac{\log z}{\log u}\left(\frac{1+\frac{0.841}{(22)^3}}{1-\frac{0.841}{(22)^3}}\right)<\frac{\log z}{\log u}\left(1+1.6\cdot 10^{-4}\right).
    \end{align*}
    This completes the proof.
\end{proof}
\end{section*}

\newpage

\printbibliography

\end{document}